\documentclass[a4paper,centertags,oneside,12pt]{amsart}
\usepackage{mathrsfs}
\usepackage{amssymb}
\usepackage{fancyhdr}
\usepackage{charter}
\usepackage{typearea}
\usepackage{pdfsync}
\usepackage{mathrsfs}
\newcommand{\comment}[1]{}
\usepackage{hyperref}
\usepackage[width=6.4in,height=8.5in]{geometry}

\newtheorem{theorem}{Theorem}[section]
\newtheorem{definition}[theorem]{Definition}
\newtheorem{proposition}[theorem]{Proposition}

\newtheorem{lemma}[theorem]{Lemma}

\newcommand{\cali}[1]{\mathscr{#1}}
\newcommand{\Cc}{\cali{C}}

\DeclareMathOperator{\supp}{supp}
\DeclareMathOperator{\codim}{codim}
\DeclareMathOperator{\Bs}{Bs}

\newcommand{\dist}{\mathop{\mathrm{dist}}\nolimits}
\newcommand{\bfs}{{\rm \mathbf{s}}} 
 
\newcommand{\loc}{ \mathrm{loc}  } 
\newcommand{\bfm}{ \mathbf{m}  } 
\newcommand{\Ac}{\cali{A}}

\newcommand{\Jc}{\cali{J}}
\newcommand{\Oc}{\cali{O}}

\newcommand{\FS}{{\rm FS}}
\newcommand{\C}{\mathbb{C}}
\newcommand{\N}{\mathbb{N}}

\renewcommand\P{\mathbb{P}}
\newcommand{\ddc}{{dd^c}}

\begin{document}

\title[Approximation  and equidistribution results for pseudo-effective  line bundles]
{Approximation  and equidistribution   results  for  pseudo-effective line bundles}

\author{Dan Coman}
\thanks{D.\ Coman is partially supported by the NSF Grant DMS-1300157}
\address{Department of Mathematics, Syracuse University, Syracuse, NY 13244-1150, USA}\email{dcoman@syr.edu}
\author{George Marinescu}
\address{Universit{\"a}t zu K{\"o}ln, Mathematisches Institut, Weyertal 86-90, 50931 K{\"o}ln, 
Deutschland   \& Institute of Mathematics `Simion Stoilow', Romanian Academy, Bucharest, Romania}
\email{gmarines@math.uni-koeln.de}
\thanks{G.\ Marinescu is partially supported by DFG funded project SFB/TRR 191}
\author{Vi{\^e}t-Anh Nguy{\^e}n}
\address{Universit\'e de Lille 1, 
Laboratoire de math\'ematiques Paul Painlev\'e, 
CNRS U.M.R. 8524,  
59655 Villeneuve d'Ascq Cedex, 
France}
\email{Viet-Anh.Nguyen@math.univ-lille1.fr}
\thanks{V.-A. Nguyen is partially supported by Vietnam Institute for Advanced Study in Mathematics (VIASM)}
\thanks{Funded through the Institutional Strategy of the University of Cologne 
within the German Excellence Initiative}
\date{December 29, 2016}

\pagestyle{myheadings}

\begin{abstract}
We study the distribution of the common zero sets of $m$-tuples 
of holomorphic sections of powers of 
$m$ singular Hermitian pseudo-effective line bundles on a compact K\"ahler manifold.
As an application, we  obtain
sufficient conditions which ensure that the wedge product of the curvature currents of 
these line bundles
can be approximated by analytic cycles.
\end{abstract}

\maketitle
\tableofcontents

\section{Introduction} \label{introduction}

 We start with a  short discussion about  the  background  of this   work.
 The remaining  of  the   section is  divided into three subsections.
 The first one deals with  the approximation  question, whereas  the second part is 
 devoted to the  equidistribution  problem. The  organization of the  article is  discussed in
 the last  subsection.
 
\subsection{Background}

Let $(X,\omega)$ be a compact K\"ahler manifold of dimension $n,$ $\dist$ be 
the distance on $X$ induced  by $\omega,$ and  $K_X$ be the {\it canonical line bundle} of $X.$  If $(L,h)$ is a  holomorphic line bundle 
on $X$  endowed  with  a singular Hermitian metric $h,$ we denote by $c_1(L,h)$ its {\it curvature current.} Recall that if $e_L$ is a holomorphic frame of $L$ 
on some open set $U\subset X$ then $|e_L|^2_h=e^{-2\phi}$, where $\phi\in L^1_{loc}(U)$ is called 
the {\it local weight} of the metric $h$ with respect to $e_L$, and $c_1(L,h)|_U=dd^c\phi$. 
Here $d=\partial+\overline\partial$,  $d^c= \frac{1}{2\pi i}(\partial -\overline\partial)$. 
We say that $h$ is {\it positively curved} (resp.\ {\it strictly positively curved}) if $c_1(L,h)\geq 0$ 
(resp.\ $c_1(L,h)\geq   \varepsilon\omega$ for some $\varepsilon>0$)  in the sense of currents. 
This is equivalent to saying that the local weights $\phi$ are plurisubharmonic (psh for short) (resp.\  strictly plurisubharmonic).
We say that  $(L,h)$ is {\it pseudo-effective} if the metric $h$ is positively curved.  
For $p\in \N$ and  $L$ a holomorphic line bundle on $X,$ let $L^p := L^{\otimes p}.$  
Given a holomorphic section $s\in   H^0(X, L^p),$  we denote
by $[s = 0]$ the current of integration (with multiplicities) over the analytic hypersurface
$\{s= 0\}\subset  X.$

Recall that a holomorphic line bundle $L$ is called \emph{big} 
if its Kodaira-Iitaka dimension equals the dimension of $X$ (see
\cite[Definition\,2.2.5]{MM07}). By the Shiffman-Ji-Bonavero-Takayama 
criterion \cite[Lemma\,2.3.6]{MM07},
$L$ is big if and only if it admits a strictly positively curved singular Hermitian  metric. 
In this  case  we also say that $(L,h)$ is  big.
 
Recall  from  \cite[Definition 1.1]{CMN15} the  following concept.
 We say that the analytic subsets $A_1,\ldots,A_m \subset X$  are 
{\it in general position} if $\codim (A_{i_1}\cap\ldots\cap A_{i_k})\geq k$ for every $1\leq k\leq m$ and $1\leq i_1<\ldots<i_k\leq m$.

\par
Let $L_k$, $1\leq k\leq m\leq n$, be $m$   holomorphic line bundles on $(X,\omega)$. 
For each $p\in\N^*,$  we define $\Ac^p(L_1,\ldots,L_m)$  to be the space of all positive closed currents $R$ of
bidegree $(m,m)$ on $X$ of the form
\begin{equation}\label{e:A^p}
 R ={1\over p^m} [s_{p1} = 0] \wedge \ldots\wedge[s_{pm} = 0]\,,\,\;s_{pj}\in H^0(X, L^p_j)\,,
\end{equation}
where $s_{pj}$ are such that the hypersurfaces $\{ s_{p1}= 0\},\ldots,\{ s_{pm} = 0\}$ are in general position.
Recall that if $L_k$ are big 
the set $\{s_{pi_1} = 0\}\cap\ldots\cap \{s_{pi_k} = 0\}$         has pure dimension $n - k$ for every $i_1 < \ldots < i_k$ in
$\{1,\ldots,m\}$ (see e.g. \cite[Proposition 3.3]{CMN15} and its proof), so $[s_{p1} = 0] \wedge \ldots\wedge[s_{pm} = 0]$ is a well-defined positive closed current of bidegree $(m,m)$
which is  equal to the current of integration
with multiplicities along the analytic set   $\{s_{p1} = 0\}\cap\ldots\cap \{ s_{pm} = 0\}$    (see e.g. \cite[Corollary 2.11,
Proposition 2.12]{D93} and \cite[Theorem 3.5]{FS95}).

When $L_1=\cdots=L_m=L,$  $\Ac^p(L,\ldots,L)$ is related with the space  $\mathcal A_m(L^p)$  introduced  by the   first and second authors in
\cite{CM13b}  as
$$
\mathcal A_m(L^p)=\left\lbrace R={1\over N}\sum_{l=1}^N  R_l:\quad  R_l\in \Ac^p(L,\ldots,L),\ N\in\N^*\quad  \right\rbrace.
$$
In particular,   $\Ac^p(L,\ldots,L)\subset \mathcal A_m(L^p).$

For each $p\in\N^*,$  we  define $\Ac^p_K(L_1,\ldots,L_m)$  to be the space of all positive closed currents $R$ of
bidegree $(m,m)$ on $X$  of the  form
\begin{equation}\label{e:A^p_K}
 R ={1\over p^m} [s_{p1} = 0] \wedge \ldots\wedge[s_{pm} = 0]\,,\,\; s_{pj}\in H^0(X, L^p_j\otimes K_X)\,,
\end{equation}
where $s_{pj}$ are such that the hypersurfaces $\{ s_{p1}= 0\},\ldots,\{ s_{pm} = 0\}$ are in general position.

\subsection{Approximation results}

Here is  our first approximation  result using the sequence of  spaces  $\Ac^p(L_1,\ldots,L_m),$ $p\geq 1.$
\begin{theorem}\label{T:main1}
Let $(X,\omega)$ be a compact K\"ahler manifold of dimension $n$ and $1\leq m\leq n$ be  an integer. For $1\leq k\leq m$ let 
$L_k$  be a holomorphic line 
bundle on $X$ endowed   with  two  singular Hermitian metrics $g_k$ and $h_k$ such that: 
\begin{itemize}
\item[(i)] $g_k$ and $h_k$  are locally bounded outside a proper analytic subset $\Sigma_k\subset X;$ 
\item[(ii)] $c_1(L_k,g_{k})\geq\varepsilon\omega$ on $X$ for some $\varepsilon>0$  and  $c_1(L_k,h_k)\geq 0$ on $X;$ 
\item[(iii)]
 $\Sigma_1,\ldots,\Sigma_m$ 
are in general position. 
 \end{itemize}
Then  there exists  a  sequence of currents $R_j\in  \Ac^{p_j}(L_1,\ldots,L_m),$ where $p_j\nearrow\infty,$
such that $R_j$ converges  weakly on $X$  to 
 $ c_1(L_1,h_1)\wedge\ldots\wedge c_1(L_m,h_m)$  as $p\to\infty\,.$ 
\end{theorem}

Working  with  sections of adjoint line bundles, i.\,e.\ 
using    the sequence of  spaces\\  $\Ac^p_K(L_1,\ldots,L_m),$ $p\geq 1,$
we obtain a more  general approximation result than  Theorem \ref{T:main1}, in the sense that the metrics $g_k$
are assumed to verify a weaker  positivity condition.
The next theorem  only deals  with two line bundles. However, it requires  a  very  weak assumption on the  sets where the metrics may not be  continuous.
\begin{theorem}\label{T:main2}
 Let $(X,\omega)$ be a compact K\"ahler manifold of dimension $n\geq2,$  and for $1\leq k\leq 2$ let 
$L_k$  be a holomorphic line 
bundle on $X$ endowed   with  two  singular Hermitian metrics $g_k$ and $h_k$ such that: 
\begin{itemize}
\item[(i)] $g_k$ and $h_k$  are  continuous outside a proper analytic subset $\Sigma_k\subset X;$ 
\item[(ii)] $c_1(L_k,h_k)\geq 0$ on $X$  and $c_1(L_k,g_{k})\geq\eta_k\omega$ on $X,$  where
 $\eta_k:\  X\to  [0,\infty)$ 
 is a function such that  for every $x\in X\setminus \Sigma_k$ there is a neighborhood
 $U_x$ of $x$ and a constant $c_x>0$ so that   
$\eta_k\geq  c_x$  on $U_x;$    
 
\item[(iii)]
 $\Sigma_1$ and $\Sigma_2$ 
are in general position. 
 \end{itemize}
 Then  there exists  a  sequence of currents $R_j\in  \Ac^{p_j}_K(L_1,L_2),$ where $p_j\nearrow\infty,$
such that $R_j$ converges  weakly on $X$  to 
 $ c_1(L_1,h_1)\wedge c_1(L_2,h_2)$  as $p\to\infty\,.$ 
\end{theorem} 

The last approximation result  deals with  several  line bundles. However, it requires  a  strong  assumption
on the  set  where the metrics may not be  continuous.
\begin{theorem}\label{T:main2bis}
 Let $(X,\omega)$ be a compact K\"ahler manifold of dimension $n$ and $1\leq m\leq n$ be  an integer. For $1\leq k\leq m$ let $L_k$  be a holomorphic line bundle on $X$ endowed   with  two  singular Hermitian metrics $g_k$ and $h_k$ such that:  
\begin{itemize}
\item[(i)] $g_k$ and $h_k$  are  continuous outside a proper analytic subset $\Sigma \subset X;$ 
\item[(ii)]  $c_1(L_k,h_k)\geq 0$ on $X$  and $c_1(L_k,g_{k})\geq\eta_k\omega$ on $X,$  where
 $\eta_k:\  X\to  [0,\infty)$ 
   is a function such that  for every $x\in X\setminus \Sigma$ there is a neighborhood
 $U_x$ of $x$ and a constant $c_x>0$ so that 
$\eta_k\geq  c_x$  on $U_x;$    
  
\item[(iii)]
 $\codim(\Sigma )\geq m.$ 
 \end{itemize}
 Then  there exists  a  sequence of currents $R_j\in  \Ac^{p_j}_K(L_1,\ldots,L_m),$ where $p_j\nearrow\infty,$
such that $R_j$ converges  weakly on $X$  to 
 $ c_1(L_1,h_1)\wedge\ldots\wedge c_1(L_m,h_m)$  as $p\to\infty\,.$ 
\end{theorem}

\subsection{Equidistribution results}

In order to  investigate the equidistribution  problem, we need to introduce some more notation and terminology.
 Let $(L_k,h_k)$, $1\leq k\leq m\leq n$, be $m$ singular Hermitian holomorphic line bundles on $(X,\omega)$. 
Let $H^0_{(2)}(X,L_k^p)$ (resp.\ $H^0_{(2)}(X,L_k^p\otimes K_X)$) be the Bergman space of $L^2$-holomorphic sections of $L_k^p $
(resp.\ of $L_k^p\otimes K_X $)   relative to the metric $h_{k,p}:=h_k^{\otimes p}$ 
induced by $h_k,$ and the metric $h^{K_X}$ on $K_X$ induced by the volume form $\omega^n$ on $X$.
These  spaces are endowed with the respective inner product
\begin{equation}\label{e:ip}
\begin{split}
(S,S')_{k,p}&:=\int_{X}\langle S,S'\rangle_{h_{k,p}}\,\omega^n\,,\;\,S,S'\in H^0_{(2)}(X,L_k^p),\\
(S,S')^K_{k,p}&:=\int_{X}\langle S,S'\rangle_{h_{k,p}\otimes h^{K_X}}\,\omega^n\,,\;\,S,S'\in H^0_{(2)}(X,L_k^p\otimes K_X).
\end{split}
\end{equation}
 For every $p\geq 1$  and $1\leq k\leq m,$ let $\sigma_{k,p}$ be   the Fubini-Study 
volume on  the projective space $\P H^0_{(2)}(X,L^p_k)$ (resp.\  $\P H^0_{(2)}(X,L^p_k\otimes K_X)$)
which is  the projectivization of the  finite-dimensional complex  spaces  $ H^0_{(2)}(X,L^p_k)$ (resp.\  $ H^0_{(2)}(X,L^p_k\otimes K_X)$) endowed  with the above 
inner product $(S,S')_{k,p}$ (resp.\  $(S,S')^K_{k,p}$). Clearly, the  measure $\sigma_{k,p}$  depends not only on $L_k$ and  $p,$
but also  on $h_k.$ 

For every $p\geq 1$ we consider the multi-projective spaces
\begin{equation}\label{e:X_p}
\begin{split}
{\mathbb X}_{p}&:= \P H^0_{(2)}(X,L^p_1)\times\ldots\times\P H^0_{(2)}(X,L^p_m),\\
{\mathbb X}_{K,p}&:= \P H^0_{(2)}(X,L^p_1\otimes K_X)\times\ldots\times\P H^0_{(2)}(X,L^p_m\otimes K_X)
\end{split}
\end{equation} 
equipped with the probability measure $\sigma_p$ which is the product of the Fubini-Study 
volumes on the components. If $S\in H^0(X,L_k^p)$ (resp.\ $S\in H^0(X,L_k^p\otimes K_X)$), we denote by $[S=0]$ the current of 
integration (with multiplicities) over the analytic 
hypersurface $\{S=0\}$ of $X$.  Set 
\[
[\bfs_p=0]:=[s_{p1}=0]\wedge\ldots\wedge[s_{pm}=0]\,,\:\:\text{for 
$\bfs_p=(s_{p1},\ldots,s_{pm})\in {\mathbb X}_{p}$ or $\in {\mathbb X}_{K,p}$,}
\]
whenever the hypersurfaces   $\{s_{p1}=0\},\ldots,\{s_{pm} = 0\}$ of $X$ are in general position. 
We also consider the probability spaces
\begin{equation}\label{e:Omega}
\begin{split}
(\Omega,\sigma_\infty)&:= \prod_{p=1}^\infty ({\mathbb X}_{p},\sigma_p),\,\\
(\Omega_K,\sigma_\infty)&:= \prod_{p=1}^\infty ({\mathbb X}_{K,p},\sigma_p)\,.
 \end{split}
\end{equation}
For the  sake of clarity  we  may write  $\Omega(h_1,\ldots,h_m),$  $\sigma_\infty(h_1,\ldots,h_m)$  (resp.\ $\Omega_K(h_1,\ldots,h_m),$ $\sigma_\infty(h_1,\ldots,h_m)$) 
in order to make  precise the dependence of $(\Omega,\sigma_\infty)$ (resp.\  $(\Omega_K,\sigma_\infty)$)  on the metrics $h_1,\ldots, h_m.$
\begin{definition}
\label{D:equidistri}
We say  that $(\Omega,\sigma_\infty)$ (resp.\ $(\Omega_K,\sigma_\infty)$)  (or simply  $\Omega$ (resp.\ $\Omega_K$)
if $\sigma_\infty$ is  clear  from the  context)   {\it equidistributes toward} a positive closed $(m,m)$ current $T$ defined on $X$ if  
   for $\sigma_\infty$-a.e. $\{\bfs_p\}_{p\geq1}\in  \Omega$ (resp.\ $\in \Omega_K$), 
we have in the weak sense of currents on $X$, 
$$\frac{1}{p^m}\,[\bfs_p=0]\to T\,\text{ as }p\to\infty\,.$$
\end{definition}
 Our first equidistribution  theorem  only deals  with two line bundles. However, it requires  a  very  weak assumption on the  sets where the metrics may not be  continuous.

 \begin{theorem}\label{T:main3}
Let $(X,\omega)$ be a compact K\"ahler manifold of dimension $n\geq 2$ and 
$(L_k,h_{k})$, $1\leq k\leq 2$, be two singular Hermitian holomorphic line 
bundles on $X$ such that:
\begin{itemize}
\item[(i)] $h_{k}$ is  continuous outside a proper analytic subset $\Sigma_{k}\subset X;$ 
\item[(ii)] $c_1(L_k,h_{k})\geq\eta_k\omega$ on $X,$  where
 $\eta_k:\  X\to  [0,\infty)$ 
  is a function such that  for every $x\in X\setminus \Sigma_k$ there is a neighborhood
 $U_x$ of $x$ and a constant $c_x>0$ so that 
$\eta_k\geq  c_x$  on $U_x;$    
  
\item[(iii)]
 $\Sigma_1$ and $\Sigma_2$ 
are in general position. 
\end{itemize}
 Then   $\Omega_K$   equidistributes toward  the  current $ c_1(L_1,h_1)\wedge  c_1(L_2,h_2).$ 
\end{theorem} 

 The last  equidistribution result  deals with  several  line bundles. However, it requires  a  strong  assumption
on the  set  where the metrics may not be  continuous.
  \begin{theorem}\label{T:main3bis}
Let $(X,\omega)$ be a compact K\"ahler manifold of dimension $n$ and 
$(L_k,h_{k})$, $1\leq k\leq m\leq n$, be $m$ singular Hermitian holomorphic line 
bundles on $X$ such that: 
\begin{itemize}
\item[(i)] $h_{k}$ is  continuous  outside a proper analytic subset $\Sigma \subset X;$ 
\item[(ii)] $c_1(L_k,h_{k})\geq\eta_k\omega$ on $X,$  where
 $\eta_k:\  X\to  [0,\infty)$ 
  is a function such that  for every $x\in X\setminus \Sigma$ there is a neighborhood
 $U_x$ of $x$ and a constant $c_x>0$ so that 
$\eta_k\geq  c_x$  on $U_x;$    
  
\item[(iii)]
 $\codim(\Sigma )\geq m.$ 
\end{itemize}
 Then   $\Omega_K$   equidistributes toward  the  current $ c_1(L_1,h_1)\wedge\ldots\wedge c_1(L_m,h_m).$ 
\end{theorem}

\subsection{Organization of the article}
  In Section \ref{S:EiA} we first
present a method which allows us to deduce approximation   results from   equidistribution theorems.
Using this  method as well as   our previous work \cite{CMN15},  
the remainder of the  section  shows that  equidistribution  theorems
(such as Theorems \ref{T:main3}, \ref{T:main3bis}, \ldots) imply  approximations theorems
(such as Theorems \ref{T:main2}, \ref{T:main2bis}, \ldots).  

The next two sections  develop necessary tools.
Section \ref{S:Bka}
studies the dimension growth of section spaces and Bergman kernel functions.
Section \ref{S:FS} establishes the convergence   towards intersection of Fubini-Study currents. Here  we apply Dinh-Sibony  equidistribution results for meromorphic transforms \cite{DS06}.

Having  these two  tools  at hands and using   the intersection theory of positive  closed  currents,
our first equidistribution theorem  for two   line bundles, i.\,e.\ Theorem \ref{T:main3}, will be proved in 
Section \ref{S:main3}. 
 
   Finally,  Section \ref{S:E_Adj} concludes the article  with  the proof of
 our second equidistribution   theorem, i.\,e.\ Theorem \ref{T:main3bis}.
 
\medskip


\section{Equidistribution implies  approximation }\label{S:EiA}

Let $(X,\omega)$  be  a  compact K\"ahler manifold  of dimension $n$ and $0\leq m\leq n$ be an  integer.
In \cite{DS09}  Dinh and Sibony    have introduced  the  following natural metric on the space   of positive closed $(m,m)$-currents on $X.$   If $R$ and $S$ are such currents,
define
$$
\dist(R,S) := \sup\limits_{\| \Phi\|_{\Cc^1}\leq 1}
|\langle R-S,\Phi\rangle|,
$$
where $\Phi$ is a smooth $(n-m, n-m)$-form on $X$ and we use the
sum of $\Cc^1$-norms of its coefficients for a fixed atlas on $X$. 

\begin{lemma}\label{L:metric}
\begin{itemize}
\item[(i)]
On the convex set of positive closed $(m,m)$-currents of mass $\leq 1$ on $X$, the topology induced by the above distance coincides with the weak
topology.
\item[(ii)] Assume that $T, (T_j)_{j=1}^\infty, (T_{jp})_{j,p=1}^\infty$ are  positive closed $(m,m)$-currents on $X$ such that
$T_j\to T$ as  $j\to\infty$ and  that  for each $j\in\N^*,$ $T_{jp}\to T_j$ as $p\to\infty.$
Then   there is a subsequence  $(p_j)_{j=1}^\infty\subset\N^*\nearrow \infty$ such that $T_{j p_j}\to T$   as $j\to\infty.$
\end{itemize}
\end{lemma}
\begin{proof}
Assertion (i) has been proved in    \cite[Proposition 2.1.4]{DS09}.

By passing to a subsequence  if necessary, we may assume  without loss of generality that  the masses of  currents $T, (T_j)_{j=1}^\infty, (T_{jp})_{j,p=1}^\infty$ are all bounded from above   by  a 
common finite constant. Therefore, 
applying assertion (i), we  obtain that
$$\lim_{j\to\infty}\dist(T_j,T)=0\quad\text{ and}\quad  \lim_{p\to\infty}\dist(T_{jp},T_j)=0\quad \text{for each}\ j\in\N^*.$$ 
The second limit  shows that for every $j\in\N^*,$  there is $p_j>p_{j-1}$ such that $\dist(T_{jp_j},T_j)\leq 1/j.$
This, combined  with the first limit, implies that  $\lim_{j\to\infty}\dist(T_{jp_j},T)=0,$  
  proving assertion (ii) in view of assertion (i).
\end{proof}
The following result  gives an  effective implication from  equidistribution problems  to approximation issues. 
\begin{proposition}
\label{P:equi_approx}
Let $(X,\omega)$ be a compact K\"ahler manifold of dimension $n$ and $1\leq m\leq n$ be  an integer. For $1\leq k\leq m$ let 
$L_k$  be a holomorphic line 
bundle on $X$ endowed   with  two  singular Hermitian metrics $g_k$ and $h_k$ such that 
\begin{itemize}
\item[(i)] $g_k$ and $h_k$  are locally bounded outside a proper analytic subset $\Sigma_k\subset X;$ 
\item[(ii)] $c_1(L_k,g_{k})\geq 0$   and  $c_1(L_k,h_k)\geq 0$ on $X;$ 
\item[(iii)]
 $\Sigma_1,\ldots,\Sigma_m$ 
are in general position;
\item[(iv)] there is a  sequence $\epsilon_j\searrow 0$  such that
$\Omega(h_1^{1\over 1+\epsilon_j} g_1^{\epsilon_j\over 1+\epsilon_j},\ldots,h_m^{1\over 1+\epsilon_j}
g_m^{\epsilon_j\over 1+\epsilon_j} )$ \\
{\rm(}resp.\ $\Omega_K(h_1^{1\over 1+\epsilon_j}g_1^{\epsilon_j\over 1+\epsilon_j},
\ldots,h_m^{1\over 1+\epsilon_j}g_m^{\epsilon_j\over 1+\epsilon_j})${\rm)} 
equidistributes towards the  current \\ 
$c_1(L_1, h_1^{1\over 1+\epsilon_j}g_1^{\epsilon_j\over 1+\epsilon_j} )
 \wedge\ldots\wedge c_1(L_m, h_m^{1\over 1+\epsilon_j}g_m^{\epsilon_j\over 1+\epsilon_j})$ 
 for all $j\in\N^*.$ 
 \end{itemize}
 Then  there exists  a  sequence of currents $R_j\in  \Ac^{p_j}(L_1,\ldots,L_m)$ 
{\rm(}resp.\  $R_j\in  \Ac^{p_j}_K(L_1,\ldots,L_m)${\rm)},  where $p_j\nearrow\infty,$
such that $R_j$ converges  weakly on $X$  to 
 $ c_1(L_1,h_1)\wedge\ldots\wedge c_1(L_m,h_m)$  as $p\to\infty\,.$ 
 \end{proposition}
\begin{proof}
We only  give the proof  when  the space $\Omega$ is considered in (iv).
The   case  of $\Omega_K$  is quite   similar, and  hence  it is  left to  the interested reader.

For  each  $1\leq k\leq m$ and $j\in\N^*,$  observe that the metric 
$h_k^{1\over 1+\epsilon_j}g_k^{\epsilon_j\over 1+\epsilon_j}$
is  locally bounded  outside $\Sigma_k$ by (i) and  that
$$
c_1(L_k, h_k^{1\over 1+\epsilon_j}g_k^{\epsilon_j\over 1+\epsilon_j} )
={ 1\over 1+\epsilon_j}\big ( c_1(L_k, h_k)+\epsilon_j\cdot c_1(L_k, g_k)\big)\geq 0
$$
by (ii).
Consequently, using  (iii) and  applying  \cite[Corollary 2.11,
Proposition 2.12]{D93} or \cite[Theorem 3.5]{FS95}, we   get that
\begin{eqnarray*}
T_j&:=& c_1(L_1,h_1^{1\over 1+\epsilon_j}g_1^{\epsilon_j \over 1+\epsilon_j})
\wedge\ldots\wedge c_1(L_m,h_m^{1\over 1+\epsilon_j}g_m^{\epsilon_j\over 1+\epsilon_j})\\
&=&
\big({ 1\over 1+\epsilon_j}\big)^m c_1(L_1,h_1)\wedge\ldots\wedge c_1(L_m,h_m)
 + \sum {\epsilon_j^{m-|J|}\over (1+\epsilon_j)^m}
 \bigwedge_{k\in J} c_1(L_k,h_k)\bigwedge_{k\in J'} c_1(L_k,g_k),
\end{eqnarray*}
the last sum being taken over all subsets $J\subsetneq\{1,\ldots,m\}$ with 
cardinal $|J|$ and $J':= \{1,\ldots,m\}\setminus J.$
Since for such a   subset $J,$  the current 
$\bigwedge_{k\in J}c_1(L_k,h_k)\bigwedge_{k\in J'} c_1(L_k,g_k)$ is a well-defined positive closed  current 
and   $\epsilon_j^{m-|J|}\searrow 0$ as $j\nearrow\infty,$
it follows that 
 \begin{equation*}
 c_1(L_1,h_1^{ 1\over 1+\epsilon_j}g_1^{\epsilon_j\over 1+\epsilon_j})\wedge\ldots\wedge
 c_1(L_m,h_m^{ 1\over 1+\epsilon_j}g_m^{\epsilon_j\over 1+\epsilon_j})\to  
 c_1(L_1,h_1)\wedge\ldots\wedge c_1(L_m,h_m)=:T\quad\text{as $j\to\infty.$}
\end{equation*}
In other words, $T_j\to T$ as $j\to\infty$.

On the  other  hand, by (iv) for each $j\in\N^*$  and each $p\in\N^*$
there is  a current $T_{jp}\in \Ac^p (L_1,\ldots,L_m)$ such that
$ T_{jp}\to T_j$ as $p\to\infty.$ Applying  Lemma \ref{L:metric} (ii) 
to the  above  family  of currents $T,$ $(T_j)_{j=1}^\infty$ and $(T_{jp})_{j,p=1}^\infty,$
we can find a subsequence  $(p_j)_{j=1}^\infty\subset\N^*\nearrow \infty$ such  that   
$$T_{j p_j}\to c_1(L_1,h_1)\wedge\ldots\wedge c_1(L_m,h_m)\quad\text{  as}\ j\to\infty. $$
Since  $R_j:=T_{j p_j}\in \Ac^{p_j} (L_1,\ldots,L_m)$  and $p_j\nearrow\infty$ as $j\nearrow\infty,$ 
the proof  is complete.
\end{proof}

To illustrate  the usefulness of  Proposition
\ref{P:equi_approx}, we give,  in the  remainder of the section, the proof of
Theorem \ref{T:main1} and Theorem \ref{T:main2}, 
Theorem \ref{T:main2bis} modulo  Theorem \ref{T:main3} and Theorem \ref{T:main3bis}.
The following  equidistribution result is needed. 

\begin{theorem}\label{T:main5}
Let $(X,\omega)$ be a compact K\"ahler manifold of dimension $n$ and 
$(L_k,h_{k})$, $1\leq k\leq m\leq n$, be $m$ singular Hermitian holomorphic line 
bundles on $X$ such that: 
\begin{itemize}
\item[(i)] $h_{k}$ is locally bounded outside a proper analytic subset $\Sigma_{k}\subset X;$ 
\item[(ii)] $c_1(L_k,h_{k})\geq\epsilon\omega$ on $X$ for some $\epsilon>0;$  
\item[(iii)]
$\Sigma_1,\ldots,\Sigma_m$ 
are in general position. 
\end{itemize}
Then $\Omega(h_1,\ldots,h_m)$ equidistributes towards 
 the current $ c_1(L_1,h_1)\wedge\ldots\wedge c_1(L_m,h_m).$ 
\end{theorem} 
\begin{proof}
When  (i) is  replaced by the following stronger  condition: 
``$h_k$ is    continuous   outside   $\Sigma_k$", the theorem 
was proved in our previous  work  \cite[Theorem 1.2]{CMN15}. 
A careful verification shows that  our proof  still works  
assuming the weaker condition (i). 
Indeed, let $P_{k,p}$ be the Bergman kernel function of $(L_k^p,h_{k,p})$ 
(see  \cite [eq.\ (4)]{CMN15}). Then the fact that 
${1\over p} \log P_{k,p}\to 0$ in $L^1(X,\omega^n)$ as $p\to\infty$, 
as well as the estimate \cite[eq.\ (14)]{CMN15} hold under the more general 
assumption (i) (see also \cite[Theorem 5.1]{CM11}). 
Moreover, \cite[Proposition 4.7]{CMN15} 
holds with the same proof for metrics that are locally bounded outside an analytic subset of $X$. \end{proof}

\comment{what we really need   in the proof of \cite [Theorem 1.2]{CMN15} is  the  estimate
$$
{1\over p} \log P_{k,p}\to 0 \quad\text{in}\ L^1(X,\omega^n)\ \text{as}\ p\to\infty,
$$
where $P_{k,p}$ is the Bergman kernel function of $(L_k^p,h_{k,p})$ (see  \cite [eq (4)]{CMN15}).
However, this inequality has been proved in \cite[Theorem 5.1]{CM11}. \end{proof}}

\smallskip

Now  we arrive at the 

\begin{proof}[Proof of Theorem \ref{T:main1}] 
Fix  a   sequence $\epsilon_j\searrow 0$ as $j\nearrow\infty.$
By (ii) we get that
$$
c_1(L_k, h_k^{1\over 1+\epsilon_j}g_k^{\epsilon_j\over 1+\epsilon_j})={1\over 1+\epsilon_j}\big(c_1(L_k, h_k)+\epsilon_j\cdot c_1(L_k, g_k)\big)\geq {\epsilon_j\epsilon \over 1+\epsilon_j}\omega,\qquad 1\leq k\leq m.
$$
Therefore, applying   Theorem \ref{T:main5} to the singular  Hermitian holomorphic big line  
bundles $(L_k,   h_k^{1\over 1+\epsilon_j}g_k^{\epsilon_j\over 1+\epsilon_j} ),$  $1\leq k\leq m,$  we  infer that
$\Omega(h_1^{1\over 1+\epsilon_j}g_1^{\epsilon_j\over 1+\epsilon_j},\ldots,h_m^{1\over 1+\epsilon_j}g_m^{\epsilon_j\over 1+\epsilon_j})$ 
   equidistributes toward  
  the  current $ c_1(L_1,h_1^{1\over 1+\epsilon_j}g_1^{\epsilon_j\over 1+\epsilon_j})\wedge\ldots\wedge c_1(L_m,h_m^{1\over 1+\epsilon_j}g_m^{\epsilon_j\over 1+\epsilon_j})$ for all $j\in\N^*.$
Putting this together    with (i) and  (iii), we are in the position to apply
Proposition
\ref{P:equi_approx}, and hence  the  proof is  complete.
\end{proof}

\medskip

\begin{proof}[Proofs of Theorem \ref{T:main2} and Theorem \ref{T:main2bis}]
 Theorem \ref{T:main2} (resp.\  Theorem \ref{T:main2bis})
 will follow  from  Theorem \ref{T:main3}  (resp.\  Theorem \ref{T:main3bis}).
We  will only give  here the proof that  Theorem \ref{T:main3} implies  Theorem \ref{T:main2}.
The  other  implication  can be proved similarly.
Fix  a   sequence $\epsilon_j\searrow 0$ as $j\nearrow\infty.$
By (ii) we get that
$$
c_1(L_k, h_k^{1\over 1+\epsilon_j}g_k^{\epsilon_j\over 1+\epsilon_j})={1\over 1+\epsilon_j}\big(c_1(L_k, h_k)+\epsilon_j\cdot c_1(L_k, g_k)\big)
\geq {\epsilon_j\eta_k \over 1+\epsilon_j}\omega,\qquad 1\leq k\leq m.
$$
Therefore, applying   Theorem \ref{T:main3} to the singular  Hermitian holomorphic   line  
bundles $(L_k,   h_k^{1\over 1+\epsilon_j}g_k^{\epsilon_j\over 1+\epsilon_j} ),$  $1\leq k\leq 2,$  we  infer that
$\Omega_K(h_1^{1\over 1+\epsilon_j}g_1^{\epsilon_j\over 1+\epsilon_j},h_2^{1\over 1+\epsilon_j}g_2^{\epsilon_j\over 1+\epsilon_j})$ 
   equidistributes toward  
  the  current $ c_1(L_1,h_1^{1\over 1+\epsilon_j}g_1^{\epsilon_j\over 1+\epsilon_j})\wedge c_1(L_2,h_2^{1\over 1+\epsilon_j}g_2^{\epsilon_j\over 1+\epsilon_j})$ for all $j\in\N^*.$
Putting this together    with (i) and  (iii), we are in the position to apply
Proposition
\ref{P:equi_approx}, and hence  the  proof  of Theorem \ref{T:main2} is  complete.
\end{proof}




\section{Dimension growth of section spaces and Bergman kernel functions}\label{S:Bka}

\par In this section we prove a theorem about the dimension growth of section spaces and the asymptotic behavior of the Bergman 
kernel function  of adjoint line bundles. 

\par Let $(L,h)$ be a singular Hermitian  holomorphic line  bundle over a compact K\"ahler manifold $(X,\omega)$ 
of dimension $n.$ Consider the space 
$H^0_{(2)}(X,L^p\otimes K_X)$ of $L^2$-holomorphic sections of $L^p$ relative to the metric 
$h_p:=h^{\otimes p}$ induced by $h,$ $h^{K_X}$ on $K_X$ and the volume form  
$\omega^n$ on $X$, endowed with the natural inner product (see \eqref{e:ip}). 
Since $H^0_{(2)}(X,L^p\otimes K_X)$ is finite dimensional, 
let
\begin{equation}\label{e:dim}
d_p:=\dim H^0_{(2)}(X,L^p\otimes K_X)-1,
\end{equation} 
and when $d_p\geq 0$ let $\{S^p_j\}_{j=0}^{d_p}$ 
be an orthonormal basis. 
We denote by $P_p$ the Bergman kernel function defined by 
\begin{equation}\label{e:Bk}
P_p(x)=\sum_{j=0}^{d_p}|S^p_j(x)|_{h_p\otimes h^{K_X}}^2,\;\;|S^p_j(x)|_{h^p\otimes h^{K_X}}^2:=
\langle S_j^p(x),S_j^p(x)\rangle_{h_p\otimes h^{K_X}},\;x\in X.
\end{equation}
Note that this definition is independent of the choice of basis.

\begin{theorem}\label{T:Bka} 
Let $(X,\omega)$ be a compact K\"ahler manifold of dimension $n,$ $(L,h)$ be a singular 
Hermitian holomorphic line bundle on $X,$ and $\Sigma \subset X$ be a proper analytic subset such that:  
\begin{itemize}
\item[(i)] $h$ is continuous outside  $\Sigma;$ 
\item[(ii)] $c_1(L,h)\geq\eta\omega$ on $X,$  where  
 $\eta:\  X\to  [0,\infty)$ is a function such that  for every $x\in X\setminus \Sigma,$ there is a neighborhood
 $U_x$ of $x$ and a constant $c_x>0$ such that 
$\eta\geq  c_x$  on $U_x.$    
\end{itemize}
For every $p\geq 1,$ let  $d_p$ be given by \eqref{e:dim}
and   $P_p$ be the Bergman kernel function defined by \eqref{e:Bk} for the space 
 $H^0_{(2)}(X,L^p\otimes K_X).$
Then  
\begin{enumerate}
\item[1)]     
 $\lim_{p\to\infty} {1\over p} \log P_p(x)=0$ locally uniformly on 
$X\setminus \Sigma$.\\
\item[2)] There is  a  constant $c>1$ such that 
 $c^{-1}\leq d_p/p^n\leq c$ for all $p\geq1$. 
\end{enumerate}
\end{theorem}

In order to prove our theorems we need  the following variant of the existence
theorem for $\overline\partial$ in the case of singular Hermitian metrics.
The smooth case goes back to Andreotti-Vesentini and H\"ormander, while the singular case was first observed by Bombieri and Skoda and proved in generality by Demailly \cite[Theorem 5.1]{D82}.

\begin{theorem}[$L^2$-estimates for $\overline\partial$]\label{T:l2}
Let $(M,\omega)$ be a K\"ahler manifold of dimension $n$  which admits a   complete  
K\"ahler  metric.
Let $(L,h)$ be a singular Hermitian holomorphic line bundle  and let $\lambda:\ M\to [0,\infty)$
be  a  continuous  function
such that $c_1(L,h)\geq \lambda\omega.$  Then for any form $g\in L_{n,1}^2(M,L,\loc)$ satisfying  
$${\overline\partial}g=0,\qquad  \int_M\lambda^{-1}|g|_{\omega,h}^2\omega^n<\infty\,, $$
 then there exists $u\in L_{n,0}^2(M,L,\loc)$ with $\overline\partial u=g$ and  
$$\int_M|u|^2_{\omega,h} \omega^n\leq \int_M\lambda^{-1}|g|_{\omega,h}^2 \omega^n\,.$$
\end{theorem}
\begin{proof}
See  \cite[Corollary 4.2]{CM13}.
\end{proof}
\medskip

\par\noindent{\em Proof of Theorem \ref{T:Bka}.}
 Following  the  arguments  of \cite{D92,CM11,CM13}  we will  first establish the  following    upper
 and lower estimates 
\eqref{e:Bke}--\eqref{e:Bke_bis} for ${\log P_p\over p}.$

To state the upper estimate \eqref{e:Bke},
let $x\in X $ and let $U_\alpha\subset X $ be a coordinate neighborhood of $x$ on which there exists a holomorphic frame $e_\alpha$ of $L$ and  $e'_\alpha$
of $K_X.$ Let $\psi_\alpha$ be a psh weight of $h$ and $\rho_\alpha$ be  a  smooth weight of $h^{K_X}$
 on $U_\alpha$. Fix $r_0>0$ so that the ball $V:=B(x,2r_0)\subset\subset U_\alpha$ and let $U:=B(x,r_0)$.
Then  \eqref{e:Bke} says that 
there exist constants $C>0$, $p_0\in\mathbb{N}$, so that  
\begin{equation}\label{e:Bke}
 {\log P_p(z)\over p} \leq\frac{\log(C r^{-2n})}{p}+2\left(\max_{B(z,r)}\psi_\alpha-\psi_\alpha(z)\right)
\end{equation}
holds for all $p>p_0$, $0<r<r_0$, and $z\in U$ with $\psi_\alpha(z)>-\infty$.  

The lower estimate \eqref{e:Bke_bis} says that for every $x\in  X\setminus \Sigma,$  there exist a constant $C=C_x,$ $p_0\in\mathbb{N}$ large  enough and an open neighborhood $U$ of $x$
such that 
\begin{equation}\label{e:Bke_bis}
-\frac{\log C}{p}\leq\frac{1}{p}\,\log P_p(z) 
\end{equation}
holds for all $p>p_0$ and $z\in U.$

\smallskip

\par  For  the upper estimate \eqref{e:Bke}, let $S\in H^0_{(2)}(X,L^p\otimes K_X)$ with $\|S\|=1$ and write $S=se_\alpha^{\otimes p}\otimes e'_\alpha$. Repeating an argument of Demailly we obtain, for   $0<r<r_0,$
\begin{eqnarray*}
|S(z)|^2_{h_p\otimes h^{K_X}}&=&|s(z)|^2e^{-2p\psi_\alpha(z)-2\rho_\alpha(z)}
\leq e^{-2p\psi_\alpha(z)-2\rho_\alpha(z)}\frac{C_1}{r^{2n}}\,\int_{B(z,r)}|s|^2\,\omega^n\\
&\leq&\frac{C_2}{r^{2n}}\,\exp\left(2p\left(\max_{B(z,r)}(\psi_\alpha+{\rho_\alpha\over p}) -\psi_\alpha(z)-{\rho_\alpha(z)\over p} \right)\right)\int_{B(z,r)}|s|^2e^{-2p\psi_\alpha-2\rho_\alpha}\,\omega^n\\
&\leq&\frac{C_3}{r^{2n}}\,\exp\left(2p\left(\max_{B(z,r)}\psi_\alpha-\psi_\alpha(z)\right)\right),
\end{eqnarray*}
where $C_1,C_2,C_3$ are constants that depend only on $x$. Hence  
$$\frac{1}{p}\,\log P_p(z)=\frac{1}{p}\,\max_{\|S\|=1}\log |S(z)|^2_{h_p\otimes h^{K_X}}\leq\frac{\log(C_3r^{-2n})}{p}+2\left(\max_{B(z,r)}\psi_\alpha-\psi_\alpha(z)\right).$$
Note that this estimate holds for all $p$ and it does not require the strict positivity of the current $c_1(L,h)$, nor the hypotheses that $X$ is compact or $\omega$ is a K\"ahler form. 
Covering $X$ by a finite  number  of such  open  set $U,$ the last estimate  implies \eqref{e:Bke}.

\smallskip

\par For the lower estimate  \eqref{e:Bke_bis}, 
let $x\in X\setminus \Sigma$ and $U_\alpha\Subset X\setminus \Sigma$ be a coordinate neighborhood of $x$ on which there exists a holomorphic frame $e_\alpha$ of $L$ and  $e'_\alpha$
of $K_X.$ Let $\psi_\alpha$ be a psh weight of $h$ and $\rho_\alpha$ be  a  smooth weight of $h^{K_X}$
 on $U_\alpha$. Fix $r_0>0$ so that the ball $V:=B(x,2r_0)\subset\subset U_\alpha$ and let $U:=B(x,r_0)$.
Next, we proceed as in \cite[Section 9]{D93b} and \cite[Theorem 4.2]{CM13} to show that there exist a constant $C=C_x>0$ and $p_0\in\mathbb{N}$ large  enough such that for all $p>p_0$ and all $z\in U$ (note that $\psi_\alpha>-\infty$ on $U$) there is a section $S_{z,p}\in H^0_{(2)}(X,L^p\otimes K_X)$ with $S_{z,p}(z)\neq0$ and 
$$\|S_{z,p}\|^2\leq C|S_{z,p}(z)|^2_{h_p\otimes h^{K_X}}\,.$$
Observe that this implies that 
$$\frac{1}{p}\,\log P_p(z)=\frac{1}{p}\,\max_{\|S\|=1}\log|S(z)|^2_{h_p\otimes h^{K_X}}\geq-\frac{\log C}{p}\,.$$

\par

 Now we   come  to the proof of Part 1).
 Observe that, by the  continuity of $\psi_\alpha$, putting \eqref{e:Bke} and \eqref{e:Bke_bis} together  implies that $\frac{1}{p}\,\log P_p\to0$ as $p\to\infty$ locally  uniformly on $V\setminus \Sigma$.  This proves Part 1).

\par Now we  turn to the proof of Part 2).
 Let $x\in X\setminus \Sigma$ and $U_\alpha\Subset X\setminus \Sigma$ be a coordinate neighborhood of $x$ on which there exists a holomorphic frame $e_\alpha$ of $L$ and  $e'_\alpha$
of $K_X.$ Let $\psi_\alpha$ be a psh weight of $h$ and $\rho_\alpha$ be  a  smooth weight of $h^{K_X}$
 on $U_\alpha$. Fix $r_0>0$ so that the ball $V:=B(x,2r_0)\subset\subset U_\alpha$ and let $U:=B(x,r_0)$.
Let $\theta\in\Cc^\infty(\mathbb R)$ be a cut-off function such that $0\leq\theta\leq1$, $\theta(t)=1$ for $|t|\leq\frac12$, $\theta(t)=0$ for $|t|\geq1$.
Define the quasi-psh function $\varphi_z$ on $X$ by
\begin{equation}\label{e:varphi_z}
\varphi_z(y)=\begin{cases}2\theta\big(\tfrac{|y-z|}{r_0}\big)\log\big(\frac{|y-z|}{r_0}\big)\,,\quad\text{for $y\in U_\alpha$}\,,\\
0,\quad\text{for $y\in X\setminus B(z,r_0)$}\,.
\end{cases}
\end{equation}
Note  the function $\varphi_z$ is psh, hence $\ddc\varphi_z\geq 0$, on $\{y:\ |y-z|\leq r_0/2\}.$
Since $V\Subset U_\alpha,$ it follows that there exists a  constant $c'>0$  such that for all
$z\in U$ we have  $\ddc \varphi_z\geq -c'\omega$ on $X$ and $\ddc \varphi_z=0$ outside $\overline{V}.$
 By assumption (ii), there is  a constant $c>0$  such that $c_1(L,h)\geq \eta\omega\geq  c\omega$ on  a neighborhood of $\overline{V}.$
Therefore,
 there exist $0<a,b<1$ and $p_0\in\N$  such that for all $p\geq p_0$ and all $z\in U$
\begin{eqnarray*}
c_1(L^p, h_pe^{-bp\varphi_z})&\geq& 0\,\qquad\text{on }X,\\
c_1(L^p, h_pe^{-bp\varphi_z})&\geq& ap\omega\quad\text{on  a neighborhood of }\overline{V}.
\end{eqnarray*}
Let $\lambda:\  X\to[0,\infty)$ be  a continuous function   such that $\lambda=ap$ on $\overline{V}$
and
$$
c_1(L^p, h_pe^{-bp\varphi_z})\geq\lambda \omega\qquad\text{on}\qquad X.
$$
 By identifying $V$  to an open ball in $\C^n,$
we may write $y=(y_1,\ldots,y_n)$ for $y\in V.$
Fix $\beta=(\beta_1,\ldots,\beta_n)\in\N^n$  with  $\sum_{j=1}^n\beta_j\leq [bp]-n.$ Let $v_{z,p,\beta}\in\Oc(V)$ be given by
\begin{equation}\label{e:v_z,p}
v_{z,p,\beta}(y):=(y_1-z_1)^{\beta_1}\cdots(y_n-z_n)^{\beta_n}\qquad\text{for}\qquad y\in V. 
\end{equation}
 Consider the form $$g_{z,p,\beta}
\in L^2_{n,1}(X,L^p),\;g_{z,p,\beta}:=\overline\partial\big(v_{z,p,\beta}\theta\big(\tfrac{|y-z|}{r_0}\big)e_\alpha^{\otimes p}\otimes e'_\alpha\big).$$
As $g_{z,p,\beta}=0$ outside $V,$ we get that
$$ \int_X{1\over\lambda}|g_{z,p,\beta}|^2_{h_p\otimes h^{K_X}}e^{-bp\varphi_z}\omega^n
=\int_V{1\over\lambda}|g_{z,p,\beta}|^2_{h_p\otimes h^{K_X}}e^{-bp\varphi_z}\omega^n
={1\over ap} \int_V |g_{z,p,\beta}|^2_{h_p\otimes h^{K_X}}e^{-bp\varphi_z}\omega^n\,.$$
Note  that  the  integral at the  right is  finite  since $\psi_\alpha$ is  bounded on $V$, $g_{z,p,\beta}=0$ on $B(z,r_0/2),$ and  $\varphi_z$ is  bounded on  $V\setminus B(z,r_0/2)$, so
\begin{eqnarray*}
\int_V|g_{z,p,\beta}|^2_{h_p\otimes h^{K_X}}e^{-bp\varphi_z}\omega^n&=&
\int_{V\setminus B(z,r_0/2)} |v_{z,p,\beta}|^2|\overline\partial\theta(\tfrac{|y-z|}{r_0})|^{2}e^{-2p\psi_\alpha}e^{-bp\varphi_z}
\omega^n\\
&\leq& C''_p\int_V|v_{z,p,\beta}|^2e^{-2p\psi_\alpha}\omega^n <\infty,
\end{eqnarray*}
where $C''_p>0$ is a constant that depends only on $x$ and $p.$

The hypotheses of Theorem \ref{T:l2} are satisfied for the  complete K\"ahler manifold
$(X,\omega)$, the semipositive  line bundle $(L^p,h_pe^{-bp\varphi_z})$ and the form
$g_{z,p,\beta},$ for all $p\geq p_0$ and 
  $z\in U$.  So by Theorem \ref{T:l2}, there exists $u_{z,p,\beta}\in L^2_{n,0}(X,L^p)$ such that $\overline\partial u_{z,p,\beta} =g_{z,p,\beta}$ and
\begin{equation}\label{e:est_u_z,p,beta}
\int_X|u_{z,p,\beta}|^2_{h_p\otimes h^{K_X}}e^{-bp\varphi_z}\omega^n\leq 
\int_X\frac{1}{\lambda}|g_{z,p,\beta}|^2_{h_p\otimes h^{K_X}}e^{-bp\varphi_z}\omega^n\leq {C''_p\over ap}
\int_V|v_{z,p,\beta}|^2e^{-2p\psi_\alpha}\omega^n.
\end{equation}
Define 
$$S_{z,p,\beta}:=v_{z,p,\beta}\theta\big(\tfrac{|y-z|}{r_0}\big)e_\alpha^{\otimes p}\otimes e'_\alpha-u_{z,p,\beta}.$$
Then 
$\overline\partial S_{z,p,\beta}=0$ and  $ S_{z,p,\beta}\in H^0_{(2)}(X,L^p\otimes K_X).$
Moreover, by \eqref{e:varphi_z} we get that
\begin{equation}\label{e:S_z,p,beta}
 S_{z,p,\beta}(y)=v_{z,p,\beta}(y)e_\alpha^{\otimes p}\otimes e'_\alpha -u_{z,p,\beta}(y)
\quad\text{for}\quad
y\in B(z,r_0/2).
\end{equation} 
Therefore, we  deduce  from this and \eqref{e:v_z,p} 
that
$\overline\partial u_{z,p,\beta}=\overline\partial S_{z,p,\beta}=0.$
Thus  $u_{z,p,\beta}$ is a $(n,0)$-holomorphic  form near $z.$

Let $\Jc$ be the sheaf of holomorphic  functions on $X$ vanishing at $z$  
and let $\bfm\subset \Oc_{X,z}$ the maximal ideal of the  ring of germs of holomorphic 
functions at $z.$
For  $k,p\in\N$ we have a  canonical residue map
$L^p\otimes K_X\to  L^p\otimes K_X\otimes (\Oc_X/\Jc^{k+1})$  which induces  in  cohomology
a map  which associates  to each global $L^2$-holomorphic   section of  $L^p\otimes K_X$
its  {\it $k$-jet} at $z:$
 $$
 J^k_p:\  H^0_{(2)}(X, L^p\otimes K_X)\to H^0\big(X, L^p\otimes K_X\otimes (\Oc_X/\Jc^{k+1})\big)
 =(L_z)^{\otimes p} \otimes (K_X)_z\otimes ( \Oc_{X,z}/\bfm^{k+1}).
 $$
 The right hand side  is  called  {\it  the space of $k$-jets of $L^2$-holomorphic sections of
$L^p\otimes K_X$ at $z.$}

Near $z$, $e^{-bp\varphi_z(y)}=r_0^{2bp}|y-z|^{-2bp}.$  
It is well-known (see \cite[p. 103]{MM07}) that for $\gamma=(\gamma_1,\ldots,\gamma_n)\in\N^n$
and $z=(z_1,\ldots,z_n)\in\C^n,$
the integral
$$
\int_{|y_1-z_1|<1,\ldots,|y_n-z_n|<1}|y_1-z_1|^{2\gamma_1}\cdots |y_n-z_n|^{2\gamma_n}|y-z|^{-2bp}\cdot i^n dy_1\wedge d\bar{y}_1\wedge \cdots\wedge dy_n\wedge d\bar{y}_n
$$
is  finite if  and only if $\sum_{j=1}^n\gamma_j\geq  [bp]-n+1.$
Putting  this  together  with \eqref{e:S_z,p,beta}, \eqref{e:v_z,p} and \eqref{e:est_u_z,p,beta} and the fact that
 $u_{z,p,\beta}$ is a $(n,0)$-holomorphic  form near $z,$ we see that the 
 $([bp]-n)$-jet of $S_{z,p,\beta}$  
coincides  with $v_{z,p,\beta}.$

Summarizing  what  has been done  so far, we have  shown that the map
$
 J^{[bp]-n}_p$ is   surjective.
 Hence, there is a constant $c>1$ such that for all $p$ sufficiently large 
 $$
 d_p=\dim H^0_{(2)}(X, L^p\otimes K_X)-1\geq  \dim  (\Oc_X/\Jc^{[bp]-n+1})-1={[bp]\choose [bp]-n}
 -1\geq  c^{-1}p^n.$$
 On the other hand, arguing as in the proof of  Siegel's lemma  \cite[Lemma 2.2.6]{MM07}, there is a constant $c>1$ such that $d_p\leq  cp^n$ for all $p\geq1$. This  completes the proof.
  $\qed$


\section{Convergence   towards intersection of Fubini-Study currents}\label{S:FS} 

\par In this section we show that the intersection of the Fubini-Study currents associated with line bundles as in Theorem \ref{T:main3} is well-defined. Moreover, we show that the sequence of wedge products of normalized Fubini-Study currents converges weakly to the wedge product of the curvature currents of $(L_k,h_k)$. We then prove that almost all zero-divisors of sections of large powers of these bundles are in general position. 

\par Let $V$ be a vector space of complex dimension $d+1$. If $V$ is endowed with a Hermitian metric, 
then we denote by $\omega_{_\FS}$ the induced Fubini-Study form on the projective space $\P(V)$ 
(see \cite[pp.\,65,\,212]{MM07}) normalized so that $\omega_{_{\FS}}^d$ is a probability measure. 
We also use the same notations for $\P(V^*)$.

\par We return to the setting of Theorem \ref{T:main3}. In fact, for the results of this section it suffices to assume that the metrics involved are only locally bounded. Namely, $(L_k,h_k)$, $1\leq k\leq m\leq n$, are singular Hermitian holomorphic line bundles on the compact K\"ahler manifold $(X,\omega)$ of dimension $n$, such that 
\begin{itemize}
\item[(i)] $h_{k}$ is {\em locally bounded} outside a proper analytic subset $\Sigma_{k}\subset X;$ 
\item[(ii)] $c_1(L_k,h_{k})\geq\eta_k\omega$ on $X,$  where  
 $\eta_k:\  X\to  [0,\infty)$ is a function such that  for every $x\in X\setminus \Sigma_k,$ there is a neighborhood
 $U_x$ of $x$ and a constant $c_x>0$ such that 
$\eta_k\geq  c_x$  on $U_x;$    
\item[(iii)]
 $\Sigma_1,\ldots,\Sigma_m$ 
are in general position. 
\end{itemize}
Consider the space $H^0_{(2)}(X,L_k^p\otimes K_X)$ of $L^2$-holomorphic sections of $L^p_k\otimes K_X$ endowed with the inner product \eqref{e:ip}.  Let 
\[
d_{k,p}:=\dim H^0_{(2)}(X,L_k^p)-1.
\]
By Part 2) of Theorem \ref{T:Bka}, there is a  constant $c>1$  such that 
\begin{equation}\label{e:dim_k}
 c^{-1}p^n\leq  d_{k,p} <cp^n .
\end{equation}
 The  {\it Kodaira map} associated  with $(L^p_k\otimes K_X,h_{k,p}\otimes h^{K_X})$  is  defined by
\begin{equation}\label{e:Kodaira_alg}
\Phi_{k,p}:\  X\dashrightarrow  {\mathbb G}(d_{k,p}, H^0_{(2)}(X,L_k^p\otimes K_X))\,,\:\:
\Phi_{k,p}(x):=\left\lbrace s\in H^0_{(2)}(X,L_k^p\otimes K_X) :\  s(x)=0  \right\rbrace,
\end{equation}
where  ${\mathbb G}(d_{k,p}, H^0_{(2)}(X,L_k^p\otimes K_X)) $ denotes the Grassmannian 
of hyperplanes in $H^0_{(2)}(X,L_k^p\otimes K_X)$  (see \cite[p.\,82]{MM07}). 
Let us identify ${\mathbb G}(d_{k,p}, H^0_{(2)}(X,L_k^p\otimes K_X))$ with $\P(H^0_{(2)}(X,L_k^p\otimes K_X)^*)$ by sending a hyperplane to an equivalence class of non-zero complex linear functionals  on $H^0_{(2)}(X,L_k^p\otimes K_X)$ having the hyperplane as their common kernel. By  composing $\Phi_{k,p}$ with this identification, we  obtain a meromorphic map  
\begin{equation}\label{e:Kodaira_alg_dual}
\Phi_{k,p}:\  X\dashrightarrow \P(H^0_{(2)}(X,L_k^p\otimes K_X)^*).
\end{equation}
To get an  analytic  description of $\Phi_{k,p},$
let 
\begin{equation}\label{e:basis}
S^{k,p}_j\in H^0_{(2)}(X,L_k^p\otimes K_X),\,j=0,\ldots,d_{k,p}\,,
\end{equation} 
be an orthonormal basis and denote by $P_{k,p}$ the Bergman kernel function of the space 
$H^0_{(2)}(X,L_k^p\otimes K_X)$ defined as in \eqref{e:Bk}.
This  basis gives identifications $H^0_{(2)}(X,L_k^p\otimes K_X)\simeq \C^{d_{k,p}+1}$ and 
$\P(H^0_{(2)}(X,L_k^p\otimes K_X)^*)\simeq \P^{d_{k,p}}$.
Let $U$ be a contractible Stein open set in $X$, let $e_k,\,e'$ be local holomorphic frames on $U$ for $L_k$, respectively $K_X$, and write $S_j^{k,p}=s_j^{k,p}e_k^{\otimes p}\otimes e',$ where  $s_j^{k,p}$ is a  holomorphic  function on $U.$
 By  composing $\Phi_{k,p}$ given in (\ref{e:Kodaira_alg_dual}) with the last  identification, 
 we  obtain a meromorphic map     $\Phi_{k,p}:X\dashrightarrow\P^{d_{k,p}}$   
which has the  following local expression
\begin{equation}\label{e:Kodaira}
\Phi_{k,p}(x)=[s_0^{k,p}(x):\ldots:s_{d_{k,p}}^{k,p}(x)]\,\text{ for }x\in U.
\end{equation}
It is  called {\it the Kodaira map defined by the basis $\{S^{k,p}_j\}_{j=0}^{d_{k,p}}$}.  

\par Next, we define the \emph{Fubini-Study currents} $\gamma_{k,p}$ of $H^0_{(2)}(X,L_k^p\otimes K_X)$ by 
\begin{equation}\label{e:FSdef}
\gamma_{k,p}|_U=\frac{1}{2}\,dd^c\log\sum_{j=0}^{d_{k,p}}|s_j^{k,p}|^2,\;   
\end{equation}
where the open set $U$ and the holomorphic functions $s_j^{k,p}$ are as  above. 
Note that $\gamma_{k,p}$ is a positive closed current of bidegree $(1,1)$ on $X$, 
and is independent of the choice of basis. 

Actually, the Fubini-Study currents are pullbacks of the Fubini-Study forms by Kodaira maps, which justifies their name. If $\omega_{_{\FS}}$ is the Fubini-Study form  on $\P^{d_{k,p}}$ then by \eqref{e:Kodaira} and \eqref{e:FSdef},
\begin{equation}\label{e:FSG}
\gamma_{k,p}=\Phi_{k,p}^*(\omega_{_{\FS}}),\quad 1\leq k\leq m.
\end{equation}  

Using \eqref{e:Bk} we introduce the psh function 
\begin{equation}\label{e:FSpot}
u_{k,p}:=\frac{1}{2p}\log\sum_{j=0}^{d_{k,p}}|s_j^{k,p}|^2=u_k+{\rho\over p} +\frac{1}{2p}\,\log P_{k,p}\,\;\text{ on }\,U\,,
\end{equation}
where   $u_k$ (resp.\ $\rho$) is the weight of the metric $h_k$ (resp.\ $h^{K_X}$) on $U$ corresponding to $e_k$ (resp.\ $e'$), i.\,e.\ $|e_k|_{h_k}=e^{-u_k}$, $|e'|_{h^{K_X}}=e^{-\rho}$.
Clearly, by \eqref{e:FSdef} and \eqref{e:FSpot},   $dd^c u_{k,p}=\frac{1}{p}\,\gamma_{k,p}$. 
Note that  ${\rho\over p} \to 0$ uniformly as $p\to\infty$ because
the metric $h^{K_X}$ is  smooth.   
Moreover, note that by \eqref{e:FSpot}, $\log P_{k,p}\in L^1(X,\omega^n)$ and 
\begin{equation}\label{e:FSB}
\frac{1}{p}\,\gamma_{k,p}=c_1(L_k,h_k)+{1\over p} c_1(K_X,h^{K_X})+\frac{1}{2p}\,dd^c\log P_{k,p}
\end{equation} 
as currents on $X$.  
For  $p\geq 1$ consider the following analytic subsets of $X$: 
$$\Sigma_{k,p}:=\left\lbrace x\in X:\,S^{k,p}_j(x)=0,\;0\leq j\leq d_{k,p}\right\rbrace,\;1\leq k\leq m\,.$$
Hence $\Sigma_{k,p}$ is the base locus of $H^0_{(2)}(X,L_k^p\otimes K_X)$, 
and $\Sigma_{k,p}\cap U=\{u_{k,p}=-\infty\}$. Note also that $\Sigma_k\cap U\supset\{u_k=-\infty\}.$
\begin{proposition}\label{P:FSwedge}
In the above hypotheses  we have the following:
\smallskip
\par
(i) For all $p$ sufficiently large and every $J\subset\{1,\ldots,m\}$ the analytic sets $\Sigma_{k,p}\,$, $k\in J$, $\Sigma_\ell\,$, $\ell\in J':=\{1,\ldots,m\}\setminus J$, are in general position. 
\smallskip
\par
(ii) If $p$ is sufficiently large then the currents 
$$\bigwedge_{k\in J}\gamma_{k,p}\;\wedge\,\bigwedge_{\ell\in J'}c_1(L_\ell,h_\ell)$$
are well defined on $X$, for every $J\subset\{1,\ldots,m\}$.
\end{proposition}
\begin{proof}  
 $(i)$ We show that for $p$ large enough, $\codim(\Sigma_{J,J',p})\geq m,$ where
$$ \Sigma_{J,J',p}:=\bigcap_{k\in J}\Sigma_{k,p}\cap\bigcap_{\ell\in J'}\Sigma_\ell .$$ The remaining assertions of $(i)$ are proved in a similar way. 
Assume for a contradiction that there exists a sequence $p_r\to\infty$ such that $\Sigma_{J,J',p_r}$ has an irreducible component $Y_r$ of dimension $n-m+s$ for some $s\geq 1$. Note that the estimate \eqref{e:Bke_bis} from the proof of Part 1) of Theorem \ref{T:Bka} holds in the case that the metric $h$ is locally bounded away from $\Sigma$. It implies that for every compact $K\subset X\setminus\Sigma_k$ there exist $c_{k,K}>0$ and $p_{k,K}\in{\mathbb N}$ such that $P_{k,p}\geq c_{k,K}$ holds on $K$ for $p\geq p_{k,K}$, where $1\leq k\leq m$. Using \eqref{e:FSpot} we infer that, given any $\epsilon$-neighborhood $V_{k,\epsilon}$ of $\Sigma_k$, $ \Sigma_{k,p_r}\subset V_{k,\epsilon}$ for all $r$ sufficiently large. Hence $Y_r\to\bigcap_{k\in J}\Sigma_k\cap\bigcap_{\ell\in J'}\Sigma_\ell=\Sigma_1\cap\cdots\cap\Sigma_m$ as $r\to\infty$. Let  $R_r=[Y_r]/|Y_r|$, where $[Y_r]$ denotes the current of integration on $Y_r$ and $|Y_r|=\int_{Y_r}\omega^{n-m+s}$. Since $R_r$ have unit mass, we may assume by passing to a subsequence that $R_r$ converges weakly to a positive closed current $R$ of bidimension $(n-m+s,n-m+s)$ and unit mass. But $R$ is supported in $\Sigma_1\cap\cdots\cap\Sigma_m $ which has dimension $\leq n-m$, so $R=0$ by the support theorem (\cite{Fed69}, see also \cite[Theorem 1.7]{Har77}), a contradiction. 

\par $(ii)$   Using $(i)$ and \cite[Corollary 2.11]{D93},  assertion $(ii)$ follows.
\end{proof}

\par The following version of Bertini's theorem is proved in \cite[Proposition 3.2]{CMN15}.

\begin{proposition}\label{P:Bertini} Let $L_k\longrightarrow X$, $1\leq k\leq m\leq n$, be holomorphic line bundles over a compact complex manifold $X$ of dimension $n$. Assume that: 
\smallskip
\par (i) $V_k$ is a vector subspace of $H^0(X,L_k)$ with basis $S_{k,0},\dots, S_{k,d_k}$, base locus $\Bs V_k:=\{S_{k,0}=\ldots=S_{k,d_k}=0\}\subset X$, such that $d_k\geq1$ and the analytic sets $\Bs V_1,\ldots,\Bs V_m$ are in general position.
\smallskip
\par (ii) $Z(t_k):=\{x\in X:\,\sum_{j=0}^{d_k}t_{k,j}S_{k,j}(x)=0\}$, where $t_k=[t_{k,0}:\ldots:t_{k,d_k}]\in\P^{d_k}$.
\smallskip
\par (iii) $\nu=\mu_1\times\ldots\times\mu_m$ is the product measure on $\P^{d_1}\times  \ldots\times\P^{d_m}$,  where $\mu_k$ is the Fubini-Study volume on $\P^{d_k}$. 
\smallskip
\par Then the analytic sets $Z(t_1),\ldots,Z(t_m)$ are in general position for $\nu$-a.e.\ $(t_1,\ldots,t_m)\in \P^{d_1}\times\ldots\times\P^{d_m}$. 
\end{proposition}

   If $\{S^{k,p}_j\}_{j=0}^{d_{k,p}}$ is an orthonormal basis of $ H^0_{(2)}(X,L_k^p\otimes K_X),$ we define the analytic hypersurface $Z(t_k)\subset X$, for $t_k=[t_{k,0}:\ldots:t_{k,d_{k,p}}]\in\P^{d_{k,p}}$, as in Proposition \ref{P:Bertini} $(ii)$. Let $\mu_{k,p}$ be the Fubini-Study volume on $\P^{d_{k,p}}$, $1\leq k\leq m$, $p\geq1$, and let $\mu_p=\mu_{1,p}\times\ldots\times\mu_{m,p}$ be the product measure on 
$\P^{d_{1,p}}\times\ldots\times\P^{d_{m,p}}$. Applying Proposition \ref{P:Bertini} we obtain:

\begin{proposition}\label{P:Bertini2}
In the above setting, if $p$ is sufficiently large then for $\mu_p$-a.e. $(t_1,\ldots,t_m)\in\P^{d_{1,p}}\times\ldots\times\P^{d_{m,p}}$ the analytic subsets $Z(t_1),\ldots,Z(t_m)\subset X$ are in general position, and $Z(t_{i_1})\cap\ldots\cap Z(t_{i_k})$ has pure dimension $n-k$ for each $1\leq k\leq m$, $1\leq i_1<\ldots<i_k\leq m$. 
\end{proposition}

\begin{proof}  Let $V_{k,p}:= H^0_{(2)}(X,L_k^p\otimes K_X)$, so $\Bs V_{k,p}=\Sigma_{k,p}$. By Proposition \ref{P:FSwedge} $(i),$ $\Sigma_{1,p},\ldots,\Sigma_{m,p}$ are in general position for all $p$ sufficiently large. We fix such $p$ and denote by $[Z(t_k)]$ the current of integration along the analytic hypersurface $Z(t_k)$; it has the same cohomology class as $pc_1(L_k,h_k)+c_1(K_X,h^{K_X})$. Proposition \ref{P:Bertini} shows that the analytic subsets $Z(t_1),\ldots,Z(t_m)$ are in general position for $\mu_p$-a.e. $(t_1,\ldots,t_m)\in\P^{d_{1,p}}\times\ldots\times\P^{d_{m,p}}$. Hence if $1\leq k\leq m$, $1\leq i_1<\ldots<i_k\leq m$, the current $[Z(t_{i_1})]\wedge\ldots\wedge[Z(t_{i_k})]$ is well defined by \cite[Corollary 2.11]{D93} and it is supported in $Z(t_{i_1})\cap\ldots\cap Z(t_{i_k})$. Since $c_1(L_k,h_k)\geq \eta_k\omega,$  it follows that
\begin{eqnarray*}
{1\over p^k}\int_X[Z(t_{i_1})]\wedge\ldots\wedge[Z(t_{i_k})]\wedge\omega^{n-k}&=&
\int_Xc_1(L_{i_1},h_{i_1})\wedge\ldots\wedge c_1(L_{i_k},h_{i_k})\wedge\omega^{n-k}+O(p^{-1})\\
&\geq& \int_X\eta_{i_1}\ldots \eta_{i_k}\omega^n +O(p^{-1})>0.
\end{eqnarray*}
So $Z(t_{i_1})\cap\ldots\cap Z(t_{i_k})\neq\emptyset$, hence it has pure dimension $n-k$. 
\end{proof}

The main result of this section is the following theorem.

\begin{theorem}\label{T:speed} We keep the hypotheses  (i), (ii), (iii) at the  beginning of the section
and use the notation introduced in \eqref{e:X_p}--\eqref{e:Omega}.
Then there exist a constant $\xi>0$ depending only on $m$ and a constant $c=c(X,L_1,h_1,\ldots,L_m,h_m)>0$ 
with the following property: For any sequence of positive numbers $\{\lambda_p\}_{p\geq1}$ with 
$$\liminf_{p\to\infty}\frac{\lambda_p}{\log p}>(1+\xi n)c,$$ 
there are subsets $E_{p}\subset {\mathbb X}_{K,p}$   such that 
\\[2pt]
(a) $\sigma_p(E_{p})\leq  c p^{\xi n} \exp(-\lambda_p/c)$ for all $p$ large enough;
\\[3pt]
(b) if $\bfs_p\in  {\mathbb X}_{K,p}\setminus E_{p}$ we have that the estimate 
$$\left|\frac{1}{p^m}\big\langle[\bfs_p=0]-\gamma_{1,p}\wedge\ldots\wedge\gamma_{m,p}\,,\phi\big\rangle\right|
\leq c\,\frac {\lambda_p}{p}\,\|\phi\|_{\Cc^2}$$
holds for every $(n-m,n-m)$ form $\phi$ of class $\Cc^2$. 

\par In particular, for $\sigma_\infty$-a.e. $\bfs\in \Omega_K$ the estimate from (b) holds for all $p$ sufficiently large.
\end{theorem}
\begin{proof}
We follow the lines of the proof of \cite[Theorem 4.2]{CMN15}
making the necessary changes. In fact, we 
apply  Dinh-Sibony's equidistribution results for meromorphic  transforms \cite{DS06}
and Propositions \ref{P:FSwedge} and \ref{P:Bertini2}.
Here the main point is that the dimension estimate \eqref{e:dim_k} 
plays the role of \cite[Proposition 4.7]{CMN15}.
\end{proof}



\section{Equidistribution for sections of two  adjoint line bundles}
\label{S:main3}

\par The main purpose of this  section  is to prove  Theorem \ref{T:main3}. Let $\gamma_{k,p}$, $k=1,2$, be the Fubini-Study currents of the spaces $H^0_{(2)}(X,L_k^p\otimes K_X)$ as defined in \eqref{e:FSdef}.

\begin{theorem}\label{T:FSwedge} In the setting of Theorem \ref{T:main3} we have 
$$\frac{1}{p^2}\,\gamma_{1,p}\wedge\gamma_{2,p}\to c_1(L_1,h_1)\wedge c_1(L_2,h_2)\;\text{ as }\,p\to\infty\,,$$ 
in the weak sense of currents on $X$.
\end{theorem}
\par Taking  for granted  the above result, we  arrive at the
\begin{proof}[Proof of Theorem \ref{T:main3}] Theorem \ref{T:main3} follows directly from 
Theorem \ref{T:speed} and  Theorem \ref{T:FSwedge}.
\end{proof}

\par The remainder of the section is  devoted  to the proof of  Theorem 
\ref{T:FSwedge}.
Let us start with the following lemma.

\begin{lemma}\label{L:FSwedge} Let $U$ be an open set in ${\mathbb C}^n$, $A,B$ be proper analytic subvarieties of $U$ with $\codim A\cap B\geq2$, and $u,v,u_p,v_p$, $p\geq1$, be psh functions on $U$ such that:

\par (i) $u$ is continuous on $U\setminus A$ and $u_p\to u$ as $p\to\infty$ locally uniformly on $U\setminus A$.

\par (ii) $v$ is continuous on $U\setminus B$ and $v_p\to v$ as $p\to\infty$ locally uniformly on $U\setminus B$.

\par (iii) the currents $dd^cu_p\wedge dd^cv_p=dd^c(u_pdd^cv_p)=dd^c(v_pdd^cu_p)$ are well defined. 

\noindent Then $dd^cu_p\wedge dd^cv_p\to dd^cu\wedge dd^cv$ in the weak sense of currents on $U\setminus(A\cap B)$. Moreover, if $n=2$ then $dd^cu_p\wedge dd^cv_p\to dd^cu\wedge dd^cv$ as measures on $U$.
\end{lemma}
 
 \begin{proof} We recall that the current $dd^c\rho\wedge T:=dd^c(\rho T)$ is well defined, where $\rho$ is a psh function and $T$ a positive closed current on $U$, if $\rho$ is locally integrable on $U$ with respect to the trace measure of $T$. The current $dd^cu\wedge dd^cv$ is well defined on $U$ since $\codim A\cap B\geq2$ and $u,v$ are locally bounded on $U\setminus A$, resp.\ on $U\setminus B$ \cite[Corollary 2.11]{D93} (see also \cite{FS95}). Since $u_p\to u$ locally uniformly on $U\setminus A$ and $u$ is continuous there, we have by \cite[Theorem 3.4]{CM11} that $u_p\to u$ in $L^1_{loc}(U)$ hence $dd^cu_p\to dd^cu$ weakly on $U$. Similarly,  $dd^cv_p\to dd^cv$ weakly on $U$. Using again the uniform convergence of $u_p$ on $U\setminus A$ and the continuity of $u$ there, it follows that $u_pdd^cv_p\to u\,dd^cv$, hence $dd^cu_p\wedge dd^cv_p\to dd^cu\wedge dd^cv$, weakly on $U\setminus A$ (see e.\,g.\  \cite{BT76,BT82}, \cite[Corollary 1.6]{D93}). Similarly one has that $v_pdd^cu_p\to v\,dd^cu$, hence $dd^cu_p\wedge dd^cv_p\to dd^cu\wedge dd^cv$, weakly on $U\setminus B$. Thus $dd^cu_p\wedge dd^cv_p\to dd^cu\wedge dd^cv$ weakly on $U\setminus(A\cap B)$. 
 
\smallskip

\par We prove now that $u_pdd^cv_p\to u\,dd^cv$ weakly on $U\setminus B$ as well. Indeed, note that by \cite[Theorem 4.1.8]{Ho} we have $u_p\to u$, $v_p\to v$ in $L^p_{loc}(U)$ for any $1\leq p<\infty$, and in the Sobolev space $W^{1,p}_{loc}(U)$ for any $1\leq p<2$. If $\chi$ is a test form supported in $U\setminus B$ then 
\begin{eqnarray*}
\int u_pdd^cv_p\wedge\chi&=&\int v_pdd^c(u_p\chi)\\
&=&\int v_pdd^cu_p\wedge\chi+\int v_p(du_p\wedge d^c\chi-d^cu_p\wedge d\chi+u_pdd^c\chi)\,.
\end{eqnarray*}
Now our claim follows since $v_pdd^cu_p\to v\,dd^cu$ weakly on $U\setminus B$ and since $v_pdu_p\to v\,du$, $v_pd^cu_p\to v\,d^cu$,  $v_pu_p\to vu$ in $L^1_{loc}(U)$. Therefore we have in fact that $u_pdd^cv_p\to u\,dd^cv$ weakly on $U\setminus(A\cap B)$. 

\smallskip

\par We consider finally the case $n=2$, so $A\cap B$ consists of isolated points. 
Let $x\in A\cap B$ and $\chi\geq0$ be a smooth function with compact support 
in $U$ so that $\chi=1$ near $x$ and $\supp\chi\cap(A\cap B)=\{x\}$. 
Since $u_pdd^cv_p\to u\,dd^cv$ weakly on 
$U\setminus(A\cap B)\supset\supp dd^c\chi$ we obtain 
$$\int\chi\,dd^cu_p\wedge dd^cv_p
=\int u_pdd^cv_p\wedge dd^c\chi\to\int u\,dd^cv\wedge dd^c\chi
=\int\chi\,dd^cu\wedge dd^cv\,.$$
Hence the sequence of positive measures $dd^cu_p\wedge dd^cv_p$ 
has locally bounded mass and any weak limit point 
$\mu$ satisfies $\mu(\{x\})=dd^cu\wedge dd^cv(\{x\})$ for $x\in A\cap B$. 
It follows that  $dd^cu_p\wedge dd^cv_p\to dd^cu\wedge dd^cv$ as measures on $U$.
\end{proof}

\begin{proof}[Proof of Theorem \ref{T:FSwedge}]  
Recall that the currents $\gamma_{1,p}\wedge\gamma_{2,p}$ and 
$c_1(L_1,h_1)\wedge c_1(L_2,h_2)$ are well defined by Proposition \ref{P:FSwedge}. 
Formula \eqref{e:FSB} implies that 
$$\frac{1}{p^2}\,\int_X\gamma_{1,p}\wedge\gamma_{2,p}\wedge\omega^{n-2}
=\int_Xc_1(L_1,h_1)\wedge c_1(L_2,h_2)\wedge\omega^{n-2} +O({1\over p})\,.$$
Hence it suffices to show that if $T$ is a limit point of the sequence 
$\left\{\frac{1}{p^2}\,\gamma_{1,p}\wedge\gamma_{2,p}\right\}$ 
then $T=c_1(L_1,h_1)\wedge c_1(L_2,h_2)$. For simplicity, we may assume that 
$\frac{1}{p^2}\,\gamma_{1,p}\wedge\gamma_{2,p}\to T$ as $p\to\infty$. 
Since $T$ and $c_1(L_1,h_1)\wedge c_1(L_2,h_2)$ have the same mass, 
it is enough to prove that $T\geq c_1(L_1,h_1)\wedge c_1(L_2,h_2)$.

\par We fix $x\in X$ and let $U$ be a neighborhood of $x$ such that there exist 
holomorphic frames $e_1$ of $L_1$, $e_2$ of $L_2$, and $e'$ of $K_X$, over $U$. 
Using the notation from 
Section \ref{S:Bka}, we let $u_1,u_2,\rho$ be the weights of $h_1,h_2, h^{K_X}$ on $U$ 
corresponding to these frames, 
and let $u_{k,p}$ be the psh functions defined in \eqref{e:FSpot}. 
Then $\frac{1}{p}\,\gamma_{k,p}=dd^cu_{k,p}$ and $c_1(L_k,h_k)=dd^cu_k$ on $U$. 
Note that $u_k$ is continuous on $U\setminus\Sigma(h_k)$. 
By \eqref{e:FSpot}  and by Part 1) of Theorem \ref{T:Bka}  and  by the smoothness of $h^{K_X}$, we have 
 $$u_{k,p}-u_k =\frac{1}{2p}\,\log P_{k,p}+{\rho\over p}\to 0,$$
 locally uniformly on $U\setminus\Sigma(h_k)$. 
 It follows by Lemma \ref{L:FSwedge} that $T=c_1(L_1,h_1)\wedge c_1(L_2,h_2)$ 
 on $U\setminus\Gamma$, and hence on $X\setminus\Gamma$, 
 where $\Gamma:=\Sigma(h_1)\cap\Sigma(h_2)$. 

\par Next we write $\Gamma=Y\cup\left(\cup_{j\geq1}Y_j\right)$, 
where $Y_j$ are the irreducible components of dimension $n-2$ and 
$\dim Y\leq n-3$. Then by Federer's support theorem (\cite{Fed69}, 
see also \cite[Theorem 1.7]{Har77}), 
$T=c_1(L_1,h_1)\wedge c_1(L_2,h_2)$ on $D=X\setminus\cup_{j\geq1}Y_j$, 
since $Y$ is an analytic subset of $D$ of dimension $\leq n-3$. 
Siu's decomposition formula (\cite{Siu74}, see also \cite[Theorem 6.19]{D93}) implies that 
\begin{equation}\label{e:Siudec}
T=R+\sum_{j\geq1}c_j[Y_j]\,,\;c_1(L_1,h_1)\wedge c_1(L_2,h_2)=R+\sum_{j\geq1}d_j[Y_j]\,,
\end{equation}
where $[Y_j]$ denotes the current of integration on $Y_j$, 
$c_j,d_j\geq0$, and $R$ is a positive closed current of bidegree $(2,2)$ on $X$ 
which does not charge any $Y_j$. To conclude the proof of Theorem \ref{T:FSwedge} 
we show that $c_j\geq d_j$ for each $j$, by using slicing as in the proof of \cite[Theorem 3.4]{CM11}.

\par Without loss of generality, let $j=1$ and $x\in Y_1$ be a regular 
point of $\Gamma$ with a neighborhood $U$ as above. 
By a change of coordinates $z=(z',z'')$ near $x$ we may assume that $x=0\in\overline\Delta^n\subset U$
and $\Gamma\cap\Delta^n=Y_1\cap\Delta^n=\{z'=0\}$, 
where $\Delta$ is the unit disk in $\mathbb C$, $z'=(z_1,z_2)$,  
$z''=(z_3,\dots,z_n)$. Let $\chi_1(z')\geq0$ (resp.\  $\chi_2(z'')\geq0$) 
be a smooth function with compact support in $\Delta^2$ 
(resp.\ in $\Delta^{n-2}$) so that $\chi_1=1$ near $0\in{\mathbb C}^2$ 
(resp.\  $\chi_2=1$ near $0\in{\mathbb C}^{n-2}$), 
and let $\beta=i/2\sum_{j=3}^ndz_j\wedge d\overline z_j$ 
be the standard K\"ahler form in ${\mathbb C}^{n-2}$. We set 
$$u_{k,p}^{z''}(z')=u_{k,p}(z',z'')\;,\;\;u_k^{z''}(z')=u_k(z',z'')\,.$$ 
Let $\Sigma_{k,p}$ denote the base locus of $H^0_{(2)}(X,L_k^p\otimes K_X)$ 
and set $\Sigma_p=\Sigma_{1,p}\cap\Sigma_{2,p}$. Then $\Sigma_{k,p}\cap U=\{u_{k,p}=-\infty\}$. 
Since $u_{k,p}\to u_k$ locally uniformly on $U\setminus\Sigma(h_k)$ and $u_k$ 
is continuous there, it follows that 
$\Sigma_p\cap\Delta^n\subset\{(z',z'')\in\Delta^n:\,|z'|<1/2\}$ for all $p$ sufficiently large. 
Thus for each $z''\in\Delta^{n-2}$ the analytic set 
$\{z'\in\Delta^2:\,(z',z'')\in\Sigma_p\cap\Delta^n\}$ is compact, hence finite, 
so the measures $dd^cu_{1,p}^{z''}\wedge dd^cu_{2,p}^{z''}$ 
are well defined \cite[Corollary 2.11]{D93}. Moreover, 
$$\mu_p^{z''}:=dd^cu_{1,p}^{z''}\wedge dd^cu_{2,p}^{z''}\to\mu^{z''}
:=dd^cu_1^{z''}\wedge dd^cu_2^{z''}$$ 
weakly as measures on $\Delta^2$ by Lemma \ref{L:FSwedge}. 
One has the slicing formula (see e.\,g.\  \cite[formula (2.1)]{DS06})
$$\int_{\Delta^n}\chi_1(z')\chi_2(z'')dd^cu_{1,p}\wedge dd^cu_{2,p}
\wedge\beta^{n-2}=
\int_{\Delta^{n-2}}\left(\int_{\Delta^2}\chi_1(z')\,d\mu_p^{z''}(z')\right)\chi_2(z'')\beta^{n-2},$$
and similarly for $dd^cu_1\wedge dd^cu_2$.
Since $dd^cu_{1,p}\wedge dd^cu_{2,p}\to T$ it follows from Fatou's lemma that 
\begin{eqnarray*}
\int_{\Delta^n}\chi_1(z')\chi_2(z'')T\wedge\beta^{n-2}&\geq&
\int_{\Delta^{n-2}}\lim_{p\to\infty}\left(\int_{\Delta^2}\chi_1(z')\,d\mu_p^{z''}(z')\right)\chi_2(z'')\beta^{n-2}\\
&=&\int_{\Delta^{n-2}}\left(\int_{\Delta^2}\chi_1(z')\,d\mu^{z''}(z')\right)\chi_2(z'')\beta^{n-2}\\
&=&\int_{\Delta^n}\chi_1(z')\chi_2(z'')dd^cu_1\wedge dd^cu_2\wedge\beta^{n-2}.
\end{eqnarray*}
This implies that $c_1\geq d_1$, since by \eqref{e:Siudec}, $T=R+c_1[z'=0]$ and $dd^cu_1\wedge dd^cu_2=R+d_1[z'=0]$ on $\Delta^n$. 
\end{proof}


\section{Equidistribution  for sections of several  adjoint line bundles}\label{S:E_Adj}

We prove here Theorem \ref{T:main3bis}.
We will need the following local property of the complex Monge-Amp\`ere operator:

\medskip

\begin{proposition}\label{P:locMA} Let $U$ be an open set in ${\mathbb C}^n$, 
$\Sigma$ be a proper analytic subset of $U$, and $u_1,\ldots, u_m$ be psh functions on $U$ 
which are continuous  on $U\setminus\Sigma$. Assume that  $\dim\Sigma\leq n-m$ 
and that $u_{k,p}\,$, where $1\leq k\leq m$ and $p\geq1$, 
are psh functions on $U$ so that $u_{k,p}\to u_k$ locally uniformly on $U\setminus\Sigma$. Then 
the currents $dd^cu_{1,p}\wedge\ldots\wedge dd^cu_{m,p}$ 
are well defined on $U$ for $p$ sufficiently enough, and 
$dd^cu_{1,p}\wedge\ldots\wedge dd^cu_{m,p}\to
dd^cu_1\wedge\ldots \wedge dd^cu_m
 $ weakly as $p\to\infty$ in the sense of currents on $U$.
\end{proposition}
\begin{proof}
It follows along the  same lines as those given in the proof of \cite[Theorem 3.4]{CM11}.
\end{proof}

\begin{proof}[Proof of Theorem \ref{T:main3bis}]
 Let $U\subset X$ be a contractible Stein open set, $u_{k,p}$, $u_k$ 
be the psh functions defined in \eqref{e:FSpot}, so $dd^cu_k=c_1(L_k,h_k)$ and 
$dd^cu_{k,p}=\frac{1}{p}\,\gamma_{k,p}$ on $U$. By    Part 1) of Theorem \ref{T:Bka}
we have that $\frac{1}{p}\,\log P_{k,p}\to 0$  locally uniformly on $U\setminus\Sigma,$
hence by \eqref{e:FSpot}, $u_{k,p}\to u_k$ locally uniformly on $U\setminus\Sigma$  as $p\to\infty$, 
for each $1\leq k\leq m$. Therefore, Proposition \ref{P:locMA} implies that
$dd^cu_{1,p}\wedge\ldots\wedge dd^cu_{m,p}\to dd^cu_1\wedge\ldots\wedge dd^cu_m$ 
weakly on $U$ as $p\to\infty$.  Thus,
we have shown that 
$$\frac{1}{p^m}\,\gamma_{1,p}\wedge\ldots\wedge\gamma_{m,p}
\to c_1(L_1,h_1)\wedge\ldots\wedge c_1(L_m,h_m)$$ as $p\to\infty$, 
in the weak sense of currents on $X$.
This,  combined  with Theorem \ref{T:speed}, implies Theorem \ref{T:main3bis}.
\end{proof}

%



\begin{thebibliography}{XXXXX}

\bibitem[A]{Al81}
H.\ Alexander,  \emph{Projective capacity}, Recent developments in several complex variables 
(Proc.\ Conf., Princeton Univ., Princeton, N.\  J., 1979), pp.\  3--27, Ann.\  of Math.\  Stud., 100, 
Princeton Univ.\  Press, Princeton, N.\ J., 1981. 

\bibitem[BT1]{BT76} E.\ Bedford and B.\ A.\ Taylor, {\em The Dirichlet problem for a complex Monge-Amp\`ere equation}, Invent.\ Math.\ {\bf 37} (1976), 1--44.

\bibitem[BT2]{BT82} E.\ Bedford and B.\ A.\ Taylor, {\em A new capacity for plurisubharmonic functions}, Acta Math.\ {\bf 149} (1982), 1--40.

\bibitem[CM1]{CM11} D.\ Coman and G.\ Marinescu, 
{\em Equidistribution results for singular metrics on line bundles}, 
Ann.\ Sci.\ \'Ec.\ Norm.\ Sup\'er.\ (4) \textbf{48} (2015), 497--536.

\bibitem[CM2]{CM13} D.\ Coman and G.\ Marinescu,
{\em Convergence of Fubini-Study currents for orbifold line bundles}, 
Internat.\ J.\ Math.\ {\bf 24} (2013), 1350051, 27 pp.

\bibitem[CM3]{CM13b} D.\ Coman and G.\ Marinescu,
{\em On the approximation of positive closed currents 
on compact K\"ahler manifolds}, 
Math.\ Rep.\ (Bucur.) {\bf 15 (65)}, No.\ 4, 373--386.

\bibitem[CMM]{CMM14} D.\ Coman and X.\ Ma and G.\ Marinescu,
{\em Equidistribution for sequences of line bundles on normal 
K\"ahler spaces}, Geometry and Topology, to appear (arXiv:1412.8184). 

\bibitem[CMN]{CMN15} D.\ Coman and  G.\ Marinescu and V.-A.\ Nguy\^en,
{\em H\"older singular metrics on big line bundles and equidistribution},  Int. Math. Res. Notices {\bf 2016}, no. 16, 5048--5075.

\bibitem[D1]{D82} J.\ P.\ Demailly, {\em Estimations $L^2$ pour l'op\'erateur 
$\overline\partial$ d'un fibr\'e holomorphe semipositif au--dessus d'une vari\'et\'e k\"ahl\'erienne compl\`ete}, 
Ann.\ Sci.\ \'Ecole Norm.\ Sup.\ {\bf 15} (1982), 457--511.


\bibitem[D2]{D90} J.-P.\ Demailly, {\em Singular Hermitian metrics on positive line bundles}, 
in {\em  Complex algebraic varieties (Bayreuth, 1990)}, Lecture Notes in Math.\ 1507, 
Springer, Berlin, 1992, 87--104.

\bibitem[D3]{D92} J.-P.\ Demailly, {\em Regularization of closed positive currents and intersection theory},
 J.\ Algebraic Geom.\ {\bf 1} (1992), 361--409.

\bibitem[D4]{D93} J.-P.\ Demailly, {\em Monge-Amp\`ere operators, Lelong 
numbers and intersection theory}, in {\em Complex analysis and geometry}, 
Plenum, New York, 1993, 115--193.

\bibitem[D5]{D93b} J.-P.\ Demailly, {\em A numerical criterion for very ample line bundles}, 
J.\ Differential Geom.\ {\bf 37} (1993), 323--374.

\bibitem[DMM]{DMM14}  T.-C.\ Dinh and X. Ma and G. Marinescu, {\em Equidistribution and  convergence  
speed for zeros of holomorphic sections of singular Hermitian line bundles}, 
J.\ Funct.\ Anal.\ \textbf{271} (2016), no.\ 11, 3082--3110.

\bibitem[DMN]{DMN15}  T.-C.\ Dinh and X. Ma and  V.-A.\ Nguyen, {\em Equidistribution   
speed for Fekete points  associated with an  ample line bundle}, 
preprint 2015, arXiv:1505.08050, 33 pages,   Ann.\ Sci.\ \'Ec.\ Norm.\ Sup\'er.\ (to appear).

\bibitem[DMS]{DMS} T.-C.\ Dinh and G.\ Marinescu and V.\ Schmidt, {\em Asymptotic distribution 
of zeros of holomorphic sections in the non-compact setting}, J.\ Stat.\ Phys.\ {\bf 148} (2012), no.\ 1, 113--136.

\bibitem[DS]{DS06} T.-C.\ Dinh and N. Sibony, {\em Distribution des valeurs de transformations 
m\'eromorphes et applications}, Comment.\ Math.\ Helv.\ {\bf 81} (2006), no.\ 1, 221--258.

 \bibitem[DS2]{DS09}  T.-C.\ Dinh and N. Sibony, 
  {\em Super-potentials of positive closed currents, intersection theory and dynamics,}  Acta Math. {\bf 203} (2009), no. 1, 1–-82.
  
 \bibitem[F]{Fed69} H.\ Federer, {\em Geometric Measure Theory}, Die Grundlehren der mathematischen Wissenschaften, Band 153, Springer-Verlag New York Inc., New York, 1969, xiv+676 pp.
 

\bibitem[FS]{FS95} J.-E.\ Forn\ae ss and N. Sibony, {\em Oka's inequality for currents 
and  applications}, Math. Ann. {\bf 301} (1995), 399--419.

 
\bibitem[Ha]{Har77} R. Harvey, {\em Holomorphic chains and their boundaries}, Several Complex Variables (Proc. Sympos. Pure Math., Vol. XXX, Part 1), 309--382, Amer. Math. Soc., 1977.
  

\bibitem[Ho]{Ho} L. H\"ormander, {\em Notions of Convexity}, Reprint of the 1994 edition. Modern Birkh\"auser Classics. Birkh\"auser Boston, Inc., Boston, MA, 2007. viii+414 pp. 


\bibitem[GZ]{GZ05}
V.\ Guedj and A.\ Zeriahi, {\em Intrinsic capacities on compact K\"ahler manifolds}, 
J.\ Geom.\ Anal.\ {\bf 15} (2005), 607--639.

 

\bibitem[MM]{MM07} X. Ma and G. Marinescu, {\em Holomorphic Morse Inequalities and Bergman Kernels},
Progress in Mathematics, 254. Birkh\"auser Verlag, Basel, 2007. xiv+422 pp.

 \bibitem[N]{Nad89}
A. M. Nadel, \emph{Multiplier ideal sheaves and K\"ahler-Einstein metrics of positive scalar
curvature}, Annals of Math. \textbf{132} (1990), 549--596.
 
\bibitem[Sh1]{Sh15a}  G. Shao, {\em  Equidistribution of Zeros of Random 
Holomorphic Sections for Moderate Measures},  Math. Z. \textbf{283} (2016),  791--806.

 
\bibitem[Sh2]{Sh15b}  G. Shao, 
{\em Equidistribution on Big Line Bundles with Singular Metrics for Moderate Measures},  
preprint 2015, arXiv:1510.09121, 29 pages,  J. Geom. Anal.  (to appear).



\bibitem[Si]{Siu74} Y. T. Siu, {\em Analyticity of sets associated to Lelong numbers and the extension of closed positive currents}, Invent. Math. {\bf 27} (1974), 53--156. 

\end{thebibliography}
\end{document}